\documentclass[a4paper,12pt]{amsart}
\usepackage[utf8x]{inputenc}
\usepackage[T1]{fontenc}
\usepackage{latexsym,amsfonts,amssymb,mathrsfs,stmaryrd}
\usepackage{amsmath,amsthm,amsrefs}
\usepackage[all]{xy}
\usepackage{color,ifpdf}
\usepackage{enumerate}
\ifpdf
\usepackage[pdftex]{graphicx}
\usepackage[pdftex,
colorlinks=true,
linkcolor=red,
citecolor=green,
hyperindex,
plainpages=false,
bookmarks=true,
bookmarksopen=true,
bookmarksnumbered,]{hyperref}
\else
\usepackage[dvipdfm]{graphicx}
\usepackage[dvipdfm,
colorlinks=true,
linkcolor=red,
citecolor=green,
hyperindex,
plainpages=false,
bookmarks=true,
bookmarksopen=true,
bookmarksnumbered,]{hyperref}
\fi

\theoremstyle{plain}
\newtheorem{theorem}{Theorem}[section]
\newtheorem*{theorem*}{Theorem 2.8}

\newtheorem{lemma}[theorem]{Lemma}
\newtheorem{proposition}[theorem]{Proposition}

\theoremstyle{definition}
\newtheorem{definition}[theorem]{Definition}
\theoremstyle{remark}

\numberwithin{equation}{section}

\renewcommand{\H}{\mathcal{H}}

\newcommand{\R}{\mathbb{R}}
\renewcommand{\i}{\subset}
\newcommand{\sm}{\setminus}

\newcommand{\midd}{\;\middle|\;}
\newcommand{\diam}{{\rm diam}}
\newcommand{\dist}{{\rm dist}}
\newcommand{\Lip}{{\rm Lip}}
\newcommand{\id}{{\rm id}}

\begin{document}
\title{Existence of Minimizers for the Reifenberg Plateau problem}
\author{Yangqin FANG}
\address{Yangqin FANG\\
Laboratoire de Math\'ematiques d'Orsay\\
Universit\'e Paris-Sud (XI)\\
91405, Orsay Cedex, France}

\email{yangqin.fang@math.u-psud.fr}

\date{}
\maketitle
\begin{abstract}
	That is, given a compact set $B \subset \R^n$ (the boundary) and a subgroup $L$
	of the \v{C}ech homology group $\check{H}_{d-1}(B;G)$
	of dimension $d$ over some commutative group $G$, 
	we find a compact set $E \supset B$
	such that the image of $L$ by the natural map
	$\check{H}_{d-1}(B;G)\to\check{H}_{d-1}(S;G)$ induced by the inclusion
	$B \to E$, is reduced to $\{ 0 \}$, and such that the Hausdorff measure
	$\H^{d}(E \setminus B)$ is minimal under these constraints. Thus we have
	no restriction on the group $G$ or the dimensions $0 < d < n$.
\end{abstract}
\section{Introduction}
Plateau problems usually concern the existence of {\em surfaces} 
that minimize an {\em area} under some {\em boundary} constraints,
but many different meanings can be given to the terms ``surface'' and 
``area'', and many different boundary constraints can be considered.

In the present text, we shall prove an existence result for a minor
variant of the homological Plateau problem considered by
Reifenberg \cite{Reifenberg:1960}. That is, we shall give ourselves
dimensions $0 < d < n$,  a compact set $B \i \R^n$, a commutative group
$G$, and a subgroup $L$ of the \v{C}ech homology group $\check{H}_{d-1}(B;G)$,
and we shall find a compact set $E \supset B$ that minimizes the Hausdorff
measure $\H^{d}(E \setminus B)$ under the constraint that 
the restriction to $L$ of the natural map
$\check{H}_{d-1}(B;G)\to\check{H}_{d-1}(S;G)$ induced by the inclusion
$B \to E$, is trivial. See the slightly more precise definitions below.

This problem was first studied by Reifenberg \cite{Reifenberg:1960},
who gave a general existence result when the group $G$ is compact.

Also, Almgren  \cite{Almgren:1968} announced an extension of Reifenberg's 
result, obtained in connection to varifolds, and where the Hausdorff measure 
$\H^{d}$ is no longer necessarily minimized alone, but integrated against an 
elliptic integrand.

More recently, De Pauw \cite{Pauw:2007a}
proved the  existence of minimizers also when
$G = {\mathbb{Z}}$ is the group of integers, $n=3$, $d=2$, and
$B$ is a nice curve. 

Here we remove these restrictions, and also use a quite different
method of proof, based on a construction of quasiminimal sets
introduced by Feuvrier \cite{Feuvrier:2009}.

Let us introduce some notation and definitions, and then we will
rapidly discuss our main result and its background. 
When $B \i \R^n$ is a compact set, $G$ is a commutative group,
and $k \geq 0$ is an integer, we shall denote by $H_{k}(B;G)$ 
and $\check{H}_{k}(B;G)$ the singular and \v{C}ech homology groups on $B$,
of order $k$ and with the group $G$;
we refer to \cite{ES:1952} for a definition and basic properties.

If $S$ is another compact set that contains $B$, we shall denote by 
$i_{B,S} : B \to S$ the natural inclusion, by
$H_{k}(i_{B,S}) : H_{k}(B;G)\to H_{k}(S;G)$
the corresponding homomorphism between homology groups, and by
$\check{H}_{k}(i_{B,S}):\check{H}_{k}(B;G)\to \check{H}_{k}(S;G)$
the corresponding homomorphism between \v{C}ech homology groups.

\begin{definition}
	Fix a compact set $B \i \R^n$, an integer $0 < d < n$, a commutative group $G$,
	and a subgroup $L$ of $\check{H}_{d-1}(B;G)$. 
	We say that the compact set $S \supset B$ spans $L$ in \v{C}ech homology if 
	$L\subset \ker\check{H}_{d-1}(i_{B,S})$.
\end{definition}

A simple case is when $L$ is the full group $\check{H}_{k}(B;G)$; then
$S \supset B$ spans $L$ in \v{C}ech homology precisely when the mapping
$H_{k}(i_{B,S})$ is trivial. But it may be interesting to study other other subgroups
$L$, and this will not make the proofs any harder.

We have a similar definition of ``$S \supset B$ spans $L$ in singular homology'',
where we just replace $\check{H}_{d-1}(i_{B,S})$ with $H_{d-1}(i_{B,S})$. It would 
be very nice if our main statement was in terms of singular homology, but unfortunately
we cannot prove the corresponding statement at this time.

We shall denote by $\H^d(E)$ the $d$-dimensional Hausdorff measure
of the Borel set $E \i \R^n$. Recall that
\[
\H^{d}(E)=\lim_{\delta\to0+}\H_{\delta}^{d}(E),
\]
where 
\[
\H_{\delta}^{d}(E)=\inf\left\{ \sum_{j}\diam(U_{j})^{d}\midd
E\subset\bigcup_{j} U_{j},\ \diam(U_{j})<\delta \right\},
\]
i.e., the infimum is over all the coverings of $E$ by a countable collection of
sets $U_j$ with diameters less than $\delta$. We refer to 
\cite{Mattila:1995,Federer:1969}
for the basic properties of $\H^d$; notice incidentally
that we could also have used the spherical Hausdorff measure, or even some
more exotic variants, essentially because the competition will rather fast be restricted
to rectifiable sets, for which the two measures are equal.

For our main result, we are given $B \i \R^n$, $d\in (0,n)$, $G$, and a 
subgroup $L$ of $\check{H}_{d-1}(B;G)$, and we set
\[
\mathcal{F}=\mathcal{F}(B,G,L)=\left\{ S\subset\mathbb{R}^{n}\midd S{\text{ 
is a compact set that contains }B}\atop {\text{ and spans }L\text{ in
\v{C}ech homology}} \right\}
\]
For any $S\in \mathcal{F}$ and any Lipschitz map
$\varphi:\mathbb{R}^{n}\to\mathbb{R}^{n}$ with $\varphi\vert_{B}=\id_{B}$, we
can easily get that $\varphi(S)\in \mathcal{F}$. Indeed, 
$i_{B,\varphi(S)}=\varphi\circ i_{B,S}$,
\[
\check{H}_{d-1}(i_{B,\varphi(S)})=
\check{H}_{d-1}(\varphi)\circ\check{H}_{d-1}(i_{B,S}),
\]
so 
\[
\ker\left( \check{H}_{d-1}(i_{B,S}) \right)\subset\ker\left(
\check{H}_{d-1}(i_{B,\varphi(S)}) \right),
\]
thus $\varphi(S)$ spans $L$ in \v{C}ech homology.

We denote by $G(n,d)$ the Grassmann manifold of unoriented $d$-plane
directions in $\mathbb{R}^{n}$. An integrand is a continuous function 
$F:\mathbb{R}^{n}\times G(n,d)\to\mathbb{R}^{+}$ which is bounded, i.e. 
there exist $0<c\leq C<+\infty$ such that $c\leq F(x,\pi)\leq C$ for all 
$x\in \mathbb{R}^{n}$ and $\pi\in G(n,d)$.
For any $d$-dimensional set $E$, we suppose that $E=E_{rec}\cup E_{irr}$,
where $E_{rec}$ is $d$-rectifaible, $E_{irr}$ is $d$-irreuler. 
For any integrand $F$ and any $\kappa> 0$, we set 
\[
\mathbb{F}_{\kappa}(E)=\kappa\H^{d}(E_{irr})+
\int_{x\in E}F(x,T_{x}E_{rec})d\H^{d}(x).
\]
We shall define a class of integrands $\mathfrak{F}$, that are integrands
$F$ satisfying the following properties: For all $x\in\mathbb{R}^{n}$,
$\delta>0$, there exists
$\varepsilon(x,\cdot):\mathbb{R}^{+}\to\mathbb{R}^{+}$ with $\lim_{r\to
0}\varepsilon(x,r)=0$ such that for all $\pi\in G(n,d)$,
\[
\mathbb{F}_{\kappa}(D_{\pi,r})\leq \mathbb{F}_{\kappa}(S)+\varepsilon(x,r)r^{d},
\]
where $0<r<\delta$, $D_{\pi,r}=\pi\cap \overline{B(x,r)}$,
$S\subset\overline{B(x,r)}$
is a compact $d$-rectifiable set which cannot be mapped
into $\partial D_{\pi,r}:=\pi\cap \partial B(x,r)$ by any Lipschitz map 
$\varphi:\mathbb{R}^{n}\to\mathbb{R}^{n}$ with
$\varphi\vert_{\partial D_{\pi,r}}=\id_{\partial D_{\pi,r}}$.

It is easy to see that the class $\mathfrak{F}$ does not depend on $\kappa$. 
Almgren introduced elliptic integrands in the papers \cite[p.423]{Almgren:1974}
and \cite[p.322]{Almgren:1968}. One can easily check that all elliptic
integrands and all continuous functions $h:\mathbb{R}^{n}\to [a,b]$ with 
$0<a<b<+\infty$ are contained in the class $\mathfrak{F}$.

We set 
\[
m = m(B,G,L,F,\kappa) = \inf \left\{ \mathbb{F}_{\kappa}(S \sm B) \, ; \, S
\in \mathcal{F}(B,G,L)\right\}.
\]
As the reader may have guessed, we want to find $E \in \mathcal{F}$
such that $\mathbb{F}_{\kappa}(S \sm B) = m$. Of course the problem will only be interesting 
when $m < +\infty$, which is usually fairly easy to arrange.
We subtracted $B$ because this way we shall not need to assume that
$\H^d(B) < +\infty$, but of course if $\H^d(B) < +\infty$ we could 
replace $\mathbb{F}_{\kappa}(S \sm B)$ with $\mathbb{F}_{\kappa}(S)$ in the definition. 
Our theorem is thus the following.

\begin{theorem}
	Let the compact set $B \i 	\R^n$, a commutative group $G$, and a subgroup 
	$L$ of $\check{H}_{d-1}(B;G)$ be given. Suppose that $ m(B,G,L,F,\kappa) < +\infty$.
	Then there exists a compact set
	$E \in \mathcal{F}(B,G,L)$ such that $\mathbb{F}_{\kappa}(E \sm B) =
	m(B,G,L,F,\kappa)$.
\end{theorem}

Notice that the statement is still true when $m = +\infty$,
but not interesting. 

As was mentioned before, this theorem was proved by Reifenberg in 
\cite{Reifenberg:1960},
under the additional assumption that $G$ be compact.

A slighthy unfortunate feature of both statements is that they use
the \v{C}ech homology groups. A similar statement with the singular homology groups
would be very welcome, both because they are simpler and because connections with
the theories of flat chains and currents would be much simpler. Unfortunately, 
singular homology does not pass to the limit as nicely as \v{C}ech homology.

Reifenberg was not the first person to give beautiful results on the Plateau Problem.
Douglas \cite{Douglas:1931} gave an essentially optimal existence result for the following 
parameterization problem: given a simple closed curve $\gamma$ in $\R^n$, find a surface $E$, parameterized by the closed unit disk in the plane, so that the restriction of the parameterization to 
the unit circle parameterizes $\gamma$, and for which the area (computed with the Jacobian and
counting multiplicity) is minimal.

But the most popular way to state and prove existence results for the Plateau problem
has been through sets of finite perimeter (De Giorgi) and currents (Federer and Fleming).
In particular, Federer and Fleming \cite{FF:1960} gave a very general existence result
for integral currents $S$ whose mass is minimal under the boundary constraint
$\partial S = T$, where $T$ is a given integral current such that $\partial T = 0$.
Mass-minimizing currents also have a very rich regularity theory; we refer to
\cite{Morgan:1989} for a nice overview.

In the author's view, Reifenberg's homological minimizers often give a better 
description of soap film than mass minimizers, and they are much closer to
(the closed support of) size minimizing currents. Those are currents $S$
that minimize the quantity $Size(S)$ under a boundary constraint $\partial S = T$
as before, but where $Size(S)$ is, roughly speaking, the $\H^d$-measure
of the set where the multiplicity function that defines $S$ as an integral current
is nonzero. Thus the mass counts the multiplicity, but not the size.
We refer to \cite{Pauw:2007a} for precise definitions, and
a more detailed account of the Plateau problem for size minimizing currents.
We shall just mention two things here, in connection to the Reifenberg problem.
Figure~\ref{fig:1} depicts the support of a current which is size minimizing, but
not mass minimizing (the multiplicity on the central disk is $2$, so the mass is 
larger than the size).
\begin{figure}
	\centering
	\includegraphics[width=3in]{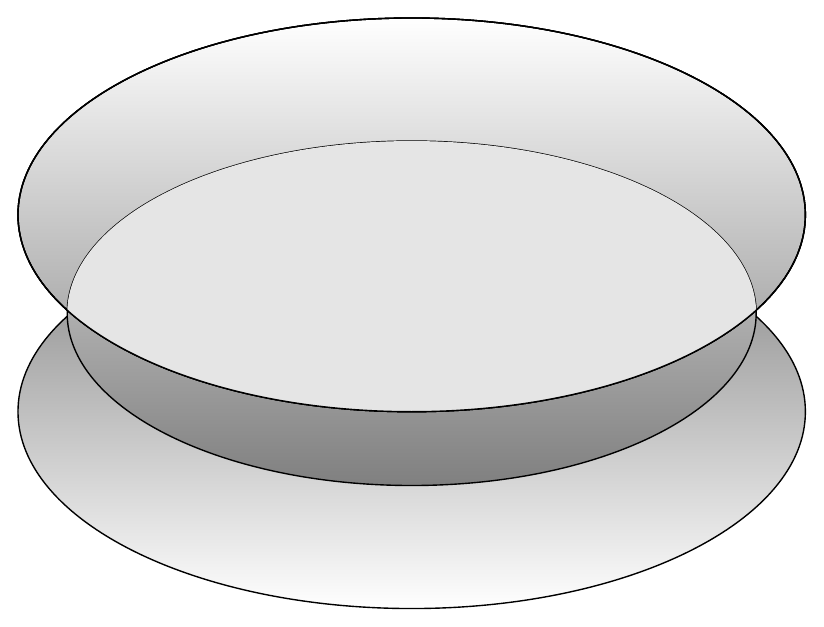}
	\caption{Size minimizing but not mass minimizing}
	\label{fig:1}
\end{figure}

Even when the boundary current $T$ is the current of integration on a smooth
(but possibly linked) curve in $\R^3$, there is no general existence for a
size minimizing current. However, Franck Morgan proved existence of a size 
minimizing current \cite{Morgan:1989} when the boundary is a smooth submanifold 
contained in the boundary of a convex body, and in \cite{PH:2003}, Thierry de 
Pauw and Robert Hardt proved the existence of currents which minimize energies 
that lie somewhere between mass and size (typically, obtained by integration 
of some small power of the multiplicity).

The reason why the usual proof of existence for mass minimizers, using a compactness
theorem, does not work for size minimizers, is that the size of $S$ does not give any
control on the multiplicity, and so the limit of a minimizing sequence may well not
have finite mass (or even not exist as currents). This issue is related to the reason why Reifenberg
restricted to compact groups (so that multiplicities don't go to infinity).

In \cite{Almgren:1968}, F. Almgren proposed a scheme for proving Reifenberg's theorem,
and even extending it to general groups and elliptic integrands. The scheme uses
the then recently discovered varifolds, or flat chains, and a multiple layers argument to
get rid of high multiplicities, but it is also very subtle and elliptic.
Incidentally, Almgren uses Vietoris relative homology groups $H_{d}^{v}$ instead of
\v{C}ech homology groups. In his paper, a boundary $B$ is
a compact $(d-1)$-rectifiable subset of $\mathbb{R}^{n}$ with 
$\H^{d-1}(B)<+\infty$, a surface $S$ is a compact $d$-rectifiable subset of
$\mathbb{R}^{n}$. For any $\sigma\in H_{d}^{v}(\mathbb{R}^{n},B;G)$, a
surface $S$ spans $\sigma$ if $i_{k}(\sigma)=0$, where we
denote by $H_{d}^{v}(\mathbb{R}^{n},B;G)$ the $d$-th Vietoris
relative homology groups of $(\mathbb{R}^{n},B)$, and 
\[
i_{k}: H_{d}^{v}(\mathbb{R}^{n},B;G)\to H_{d}^{v}(\mathbb{R}^{n},B\cup S;G)
\]
is the homomorphism induced by the inclusion map $i:B\to B\cup S$. 
But we should mention that Dowker, in \cite[Theorem 2a]{Dowker:1952}, proved that 
\v{C}ech and Vietoris homology groups over an abelian group $G$ are isomorphic
for arbitrary topological spaces. 

There is some definite relation between Refenberg's homological Plateau
problem and the size minimizing currents, and for instance
T. De Pauw \cite{Pauw:2007a} shows that in the simple case when $B$ is a curve,
the infimums for the two problems are equal. 
In the same paper, T. De Pauw also extends Reifenberg's result (for curves
in $\R^3$) to the group $G = {\mathbb Z}$. Unfortunately, even though
the proof uses minimizations among currents, this does not yet give a size minimizer 
(one would need to construct an appropriate current on the minimizing set).

\medskip
Our proof here is more in the spirit of the initial proof of Reifenberg,
but will rely on two more recent developments that make it work more smoothly and
ignore multiplicity issues.

The first development is a lemma introduced by Dal Maso,
Morel, and Solimini \cite{DMS} in the context of the Mumford-Shah 
functional, and which gives a sufficient condition, on a sequence of sets $E_k$
that converges to a limit $E$ in Hausdorff distance,
for the lower semicontinuity inequality
\[
\H^d(E) \leq \liminf_{k \to +\infty} \H^d(E_k).
\]
It is very convenient here because we want to work with sets and we do not 
want to use weak limits of currents. Since we want to deal with integrands,
we will show the following lower semicontinuity inequality,
\[
\mathbb{F}_{\kappa}(E)\leq \liminf_{k\to +\infty}\mathbb{F}_{\kappa}(E_{k}).
\]

But our main tool will be a recent result of V. Feuvrier \cite{Feuvrier:2009}, 
where he uses a construction of polyhedral networks adapted to a given set
(think about the usual dyadic grids, but where you manage to have faces that
are very often parallel to the given set) to construct a minimizing sequence
for our problem, but which has the extra feature that it is composed of
locally uniformly quasiminimal sets, to which we can apply Dal Maso,
Morel, and Solimini's lemma.

Such a construction was used by Xiangyu Liang, to prove existence results
for sets that minimize Hausdorff measure under some homological generalization
of a separation constraint (in codimension larger than $1$).


\section{Existence of minimizers under Reifenberg homological conditions}
In this section we prove an existence theorem for sets in $\mathbb{R}^{n}$
that minimize the Hausdorff measure under Reifenberg homological conditions.

\begin{definition}
	A polyhedral complex $\mathcal{S}$ is a finite set of closed convex
	polytopes in $\mathbb{R}^{n}$, such that two conditions are satisfied:
	\begin{enumerate}[{\indent\rm (1)}]
		\item If $Q\in \mathcal{S}$, and $F$ is a face of $Q$, then
			$F\in\mathcal{S}$; 
		\item If $Q_1,Q_2\in \mathcal{S}$, then $Q_{1}\cap Q_{2}$ is a face of
			$Q_1$ and $Q_2$ or $Q_1\cap Q_2=\emptyset$.
	\end{enumerate}
	The subset $\left\vert \mathcal{S}\right\vert:=\cup_{Q\in\mathcal{S}}Q$ of
	$\mathbb{R}^{n}$ equipped with the induced topology is called the underlying 
	space of $\mathcal{S}$. The $d$-skeleton of $\mathcal{S}$ is the union of the
	faces whose dimension is at most $d$.

	A dyadic complex is a polyhedral complex consisting of closed dyadic cubes.
\end{definition}
\bigskip
Let $\Omega\subset\mathbb{R}^{n}$ be an open subset, $0<M<+\infty$,
$0<\delta\leq+\infty$, $\ell\in\mathbb{N}$, $0\leq \ell \leq n$.
Let $f:\Omega\to\Omega$ be a Lipschitz map; we set 
\[
W_{f}=\{ x\in\Omega\mid f(x)\neq x \}.
\]

\begin{definition}
	Let $E$ be a relatively closed set in $\Omega$. We say that $E$ is an
	$(\Omega,M,\delta)$-quasiminimal set of dimension $\ell$ if,
	$\H^{\ell}(E\cap B)<+\infty$ for every closed ball $B\subset\Omega$, and 
	\[
	\H^{\ell}(E\cap W_{f})\leq M\H^{\ell}(f(E\cap W_{f}))
	\]
	for every Lipschitz map $f:\Omega\to\Omega$ such that $W_{f}\cup f(W_{f})$
	is relatively compact in $\Omega$ and $\diam(W\cup f(W_{f}))<\delta$.

	We denote by
	${\bf QM}(\Omega,M,\delta,\H^{\ell})$ the collection of all 
	$(\Omega,M,\delta)$-quasiminimal sets of dimension $\ell$.
\end{definition}
\medskip
We note that, for any open set $\Omega'\subset \Omega$, any positive
numbers $\delta'\leq \delta$, and any $M'\geq M$, if $E\in {\bf
QM}(\Omega,M,\delta,\H^{\ell})$, then 
$E\cap \Omega'\in {\bf QM}(\Omega',M',\delta',\H^{\ell})$.

\begin{definition}
	Let $\Omega$ be an open subset of $\mathbb{R}^{n}$. A relatively closed set
	$E\subset \Omega$ is said to be locally Ahlfors-regular of
	dimension $d$ if there is a constant $C>0$ and $r_{0}>0$ such that
	\[
	C^{-1}r^{d}\leq\H^{d}(E\cap B(x,r))\leq Cr^{d}
	\]
	for all $0<r<r_{0}$ with $B(x,2r)\subset\Omega$.
\end{definition}

\begin{lemma}\label{le:QMAR}
	Let $E$ be a $d$-rectifiable subset of $\mathbb{R}^{n}$. If $E$ is a local
	Ahlfors-regular and $\H^{d}(E)<+\infty$, then for $\H^d$-a.e. $x\in E$, $E$
	has a true tangent plane at $x$, i.e. there exists a $d$-plane $\pi$ such that for
	any $\varepsilon>0$, there is a $r_{\varepsilon}>0$ such that  
	\[
	E\cap B(x,r)\subset \mathcal{C}(x,\pi,r,\varepsilon),\ \text{ for }
	0<r<r_{\varepsilon},
	\]
	where 
	\[
	\mathcal{C}(x,\pi,r,\varepsilon)=\{ y\in B(x,r)\mid \dist(y,\pi)\leq
	\varepsilon\left\vert y-x\right\vert\}.
	\]
\end{lemma}
\begin{proof}
	Since $E$ is rectifiable, by Theorem 15.11 in~\cite{Mattila:1995}, for
	$\H^{d}$-a.e. $x\in E$, $E$ has an approximate tangent plane $\pi$ at $x$,
	i.e. 
	\[
	\limsup_{\rho\to 0}\frac{\H^{d}(E\cap B(x,\rho))}{\rho^{d}}>0,
	\]
	and there exists a $d$-plane $\pi$ such that for all $\varepsilon>0$,
	\begin{equation}\label{eq:app}
		\lim_{\rho\to 0}\rho^{-d}\H^{d}(E\cap
		B(x,\rho)\setminus\mathcal{C}(x,\pi,\rho,\varepsilon))=0.
	\end{equation}

	We will show that $\pi$ is a true tangent plane. Suppose not, that is, there 
	exists an $\varepsilon>0$ such that for all $\rho>0$, $E\cap
	B(x,\rho)\setminus\mathcal{C}(x,\pi,\rho,\varepsilon)\neq\emptyset$. We take
	a sequence of points $y_{n}\in
	E\setminus\mathcal{C}(x,\pi,\rho,\varepsilon)$ with $\left\vert
	y_{n}-x\right\vert \to 0$, we put $\rho_{n}=2\left\vert y_{n}-x\right\vert$,
	then 
	\[
	B(x,\rho_{n})\setminus\mathcal{C}\left(x,\pi,\rho_{n},\frac{\varepsilon}{2}\right)\supset
	B\left(y_{n},\frac{\varepsilon\rho_{n}}{4}\right)
	\]
	and
	\begin{align*}
		\rho_{n}^{-d}\H^{d}\left(E\cap B(x,\rho_{n})\setminus\mathcal{C}
		\left(x,\pi,\rho_{n}, \frac{\varepsilon}{2}\right)\right)
		&\geq\rho_{n}^{-d}\H^{d}\left(E\cap B\left(y_{n},
		\frac{\varepsilon\rho_{n}}{4}\right)\right)\\
		&\geq C^{-1}\left(\frac{\varepsilon}{4}\right)^{d},
	\end{align*}
	this is in contradiction with~(\ref{eq:app}), so we proved the lemma.
\end{proof}

Let $\{ E_{k} \}$ be a sequence of closed sets in $\Omega$, and $E$ a closed
set of $\Omega$. We say that $E_{k}$ converges to $E$ if 
\[
\lim_{k\to\infty}d_{K}(E,E_{k})=0 \text{ for every compact set }K\subset \Omega,
\]
where
\[
d_{K}(E,E_{k})=\sup\{ \dist(x,E_{k})\mid x\in E\cap K \}+\sup\{ \dist(x,E)\mid
x\in E_{k}\cap K\}.
\]
For any set $E\subset\mathbb{R}^{n}$, we set 
\[
E^{*}=\{ x\in E\mid \H^{d}(E\cap B(x,r))>0,\ \forall r>0  \};
\]
we call $E^{*}$ the core of $E$.
We will prove the following lower semicontinuity properties.

\begin{theorem}\label{thm:lsc}
	Let $\Omega\subset\mathbb{R}^{n}$ be an open set. Let $(E_{k})_{k\geq 1}$ be
	a sequence of quasiminimal sets in ${\bf QM}(\Omega, M,\delta,\H^{d})$ such 
	that $E_{k}=E_{k}^{\ast}$ and $E_{k}$ converges to $E$. Then for any
	$F\in\mathfrak{F}$,
	\[
	\mathbb{F}_{\kappa}(E)\leq\liminf_{k\to+\infty}\mathbb{F}_{\kappa}(E_{k}).
	\]
\end{theorem}
\begin{proof}
	We may suppose that 
	\[
	\liminf_{k\to +\infty}\mathbb{F}_{\kappa}(E_{k})<+\infty.
	\]
	In particular 
	\[
	\H^{d}(E)\leq \liminf_{k\to +\infty}\H^{d}(E_{k})\leq 
	\frac{1}{\min\{\kappa, \inf F\}}\liminf_{k\to +\infty}\mathbb{F}_{\kappa}(E_{k})<+\infty.
	\]

	We take $0<\varepsilon<\frac{1}{2}$, $\varepsilon'>0$ and $\rho\in (0,1)$
	such that $M^{2}3^{d}\varepsilon<1$, $\varepsilon'<\frac{\varepsilon}{8}$ 
	and $1-(1-\rho)^{d}<\frac{\varepsilon}{2}$.

	Applying Theorem 4.1 in~\cite{David:2003}, we get that
	$E\in {\bf QM}(\Omega, M,\delta,\H^{d})$, hence rectifiable (see
	\cite{Almgren:1976}), then by Theorem 17.6 in~\cite{Mattila:1995}, 
	for $\H^{d}$-a.e. $x\in E$, 
	\[
	\lim_{r \to 0}\frac{\H^{d}(E\cap B(x,r))}{\omega_{d}r^{d}}=1,
	\]
	where $\omega_{d}$ denote the Hausdorff measure of $d$-dimensional unit
	ball. So we can find a set
	$E'\subset E$ with $\H^{d}(E\setminus E')=0$ such that for any $x\in E'$ 
	there exists $r'(\varepsilon',x)>0$, 
	\[
	(1-\varepsilon')\omega_{d}r^{d}\leq\H^{d}(E\cap
	B(x,r))\leq(1+\varepsilon')\omega_{d}r^{d},
	\]
	for all $0<r<r'(\varepsilon',x)$.

	Then 
	\begin{align*}
		\H^{d}(E\cap B(x,r)\setminus B(x,(1-\rho)r))&\leq(1+\varepsilon')
		\omega_{d}r^{d}-(1-\varepsilon')\omega_{d}(1-\rho)^{d}r^{d}\\
		&=\frac{(1+\varepsilon')-(1-\varepsilon')(1-\rho)^{d}}{1-\varepsilon'}
		(1-\varepsilon')\omega_{d}r^{d}\\
		&\leq\left( \frac{2\varepsilon'}{1-\varepsilon'}+\left( 1-(1-\rho)^{d}
		\right) \right)\H^{d}(E\cap B(x,r))\\
		&\leq \varepsilon\H^{d}(E\cap B(x,r)).
	\end{align*}

	Since $E$ is quasiminimal, by Proposition 4.1 in~\cite{DS:2000}, we know
	that $E$ is local Ahlfors regular, since $E$ is rectifiable and
	$\H^{d}(E)<+\infty$, by lemma~\ref{le:QMAR}, we have that for $\H^{d}$-a.e.
	$x\in E$, $E$ has a tangent plane $T_{x}E$ at $x$, so we can find 
	$E''\subset E'$ with $\H^{d}(E'\setminus E'')=0$ such that for all
	$\varepsilon''>0$ and for all $x\in E''$ there exists 
	$r''(\varepsilon'',x)>0$ such that for all $0<r<r''(\varepsilon'',x)$, 
	\[
	E\cap B(x,r)\subset\mathcal{C}(x,r,\varepsilon''),
	\]
	where 
	\[
	\mathcal{C}(x,r,\varepsilon'')=\left\{ y\in \overline{B(x,r)} \midd
	\dist(y, T_{x}E)\leq \varepsilon''\left\vert x-y\right\vert\right\}.
	\]

	We consider the function
	$\psi_{\rho,r}:\mathbb{R}\to\mathbb{R}$ defined by 
	\[
	\psi_{\rho, r}(t)=
	\begin{cases}
		0,& t\leq (1-\rho)r\\
		\frac{3}{\rho r}\left( t-(1-\rho)r
		\right),&(1-\rho)r<t\leq(1-\frac{2\rho}{3})r\\
		1,& (1-\frac{2\rho}{3})r<t\leq (1-\frac{\rho}{3})r\\
		-\frac{3}{\rho r}(t-r),&(1-\frac{\rho}{3})r<t\leq r\\
		0,&t>r,
	\end{cases}
	\]
	It is easy to see that $\psi_{\rho,r}$ is a Lipschitz map with Lipschitz
	constant $\frac{3}{\rho r}$.

	We take the Lipschitz map $\varphi_{x,\rho,r}:\mathbb{R}^{n}\to\mathbb{R}^{n}$
	given by
	\[
	\varphi_{x,\rho,r}(y)=\psi_{\rho,r}(\left\vert
	y-x\right\vert)\Pi(y)+(1-\psi_{\rho,r}(\left\vert y-x\right\vert))y,
	\]
	where we denote by $\Pi:\mathbb{R}^{n}\to T_{x}E$ the orthogonal projection.
	It is easy to check that 
	\[
	\varphi_{x,\rho,r}\vert_{B(x,(1-\rho)r)}=\id_{B(x,(1-\rho)r)}
	\]
	and
	\[
	\varphi_{x,\rho,r}\vert_{B(x,r)^{c}}=\id_{B(x,r)^{c}}.
	\]

	Let $\varepsilon''$ and $h$ be such that $\varepsilon''<\frac{\rho}{3}$ and
	$0<\varepsilon''<h<\frac{\rho}{3}$, and put 
	\[
	A_{h}=\left\{ y\in \overline{B(x,r)}\midd \dist(y,T_{x}E)\leq hr \right\},
	\]
	then $\mathcal{C}(x,r;\varepsilon'')\subset A_{h}$. We will show that
	\[
	\Lip\left( \varphi_{x,\rho,r}\vert_{A_{h}}\right)\leq 2+\frac{3h}{\rho}.
	\]
	We set 
	\[
	\Pi^{\perp}(y)=y-\Pi(y),\ y\in \mathbb{R}^{n},
	\]
	then
	\[
	\left\vert\Pi^{\perp}(y) \right\vert\leq hr,\ \forall y\in A_{h}.
	\]
	For any $y_{1},y_{2}\in A_{h}$,
	\[
	\begin{split}
		\varphi_{x,\rho,r}(y_{1})-\varphi_{x,\rho,r}(y_{2})&=y_{1}-y_{2}+
		\psi_{\rho,r}(\left\vert y_{1}-x\right\vert)\Pi^{\perp}(y_{1})-
		\psi_{\rho,r}(\left\vert y_{2}-x\right\vert)\Pi^{\perp}(y_{2})\\
		&=(y_{1}-y_{2})+\psi_{\rho,r}(\left\vert
		y_{1}-x\right\vert)\left(\Pi^{\perp}(y_{1})-\Pi^{\perp}(y_{2})\right)\\
		&\qquad+\left(\psi_{\rho,r}(\left\vert y_{1}-x\right\vert)-
		\psi_{\rho,r}(\left\vert	y_{2}-x\right\vert)\right)\Pi^{\perp}(y_{2}),
	\end{split}
	\]
	thus
	\begin{align*}
		\left\vert \varphi_{x,\rho,r}(y_{1})-\varphi_{x,\rho,r}(y_{2})\right\vert&\leq \left\vert
		y_{1}-y_{2}\right\vert+\left\vert y_{1}-y_{2}\right\vert+\frac{3}{\rho
		r}\left\vert \left\vert y_{1}\right\vert-\left\vert
		y_{2}\right\vert\right\vert rh\\
		&\leq\left(2+\frac{3h}{\rho} \right)
		\left\vert y_{1}-y_{2}\right\vert,
	\end{align*}
	and we get that 
	\[
	\Lip\left( \varphi_{x,\rho,r}\vert_{A_{h}}\right)\leq 2+\frac{3h}{\rho}.
	\]

	Since $E_{k}\to E$ in $\Omega$, and $\overline{B(x,r)}\subset \Omega$ and 
	\[
	E\cap \overline{B(x,r)}\subset\mathcal{C}(x,r,\varepsilon'')\subset A_{h},
	\]
	there exist a number $k_{h}$ such that for $k\geq k_{h}$,
	\[
	E_{k}\cap \overline{B(x,r)}\subset A_{h}.
	\]

	Since 
	\[
	\varphi_{x,\rho,r}\vert_{B(x,r)^{c}}=\id_{B(x,r)^{c}}
	\]
	and
	\[
	\varphi_{x,\rho,r}(B(x,r))\subset B(x,r),
	\]
	we have that 
	\[
	\varphi_{x,\rho,r}(E_{k}\cap B(x,r))=\varphi_{x,\rho,r}(E_{k})\cap B(x,r).
	\]
	We put $r'=\left( 1-\frac{\rho}{3} \right)r$, $r''=\left( 1-\frac{2\rho}{3}
	\right)r$, $r'''=(1-\rho)r$, $\pi=T_{x}E$. Note that 
	\[
	\partial B\left( x,r'\right)\cap\pi\subset \varphi_{x,\rho,r}(E_{k})
	\]
	and 
	\[
	\varphi_{x,\rho,r}(E_{k})\cap B\left(x,r'\right)\subset B\left(x,r''\right)
	\cup \left( \left( B(x,r')\setminus B(x,r'')\right)\cap\pi \right).
	\]
	We put
	\[
	D_{\pi,r''}=\overline{B\left( x,r'' \right)}\cap\pi
	\]
	and 
	\[
	S_{k,r''}=\varphi_{x,\rho,r}(E_{k})\cap\overline{B\left(x,r''\right)}.
	\]
	We will show that for any Lipschitz mapping 
	$\varphi:\mathbb{R}^{n}\to\mathbb{R}^{n}$ which is identity on  $\partial
	D_{\pi,r''}$
	cannot map $S_{k,r''}$ into $\partial D_{\pi,r''}$.
	Suppose not, that is, there is a Lipschitz map
	$\varphi_{k}:\mathbb{R}^{n}\to\mathbb{R}^{n}$ such that
	\[
	\varphi_{k}\vert_{\partial D_{\pi,r''}}=\id_{\partial D_{\pi,r''}}
	\]
	and
	\[
	\varphi_{k}(S_{k,r''})\subset\partial D_{\pi,r''}.
	\]
	We consider the map 
	\[
	\tilde{\phi}_{k}:B(x,r')^{c}\cup\left[\left( B(x,r')
	\setminus B(x,r'')\right)\cap \pi\right]\cup B(x,r'')\to\mathbb{R}^{n}
	\]
	defined by 
	\[
	\tilde{\phi}_{k}(x)=\begin{cases}
		x,&x\in B(x,r')^{c}\cup\left[\left( B(x,r')
		\setminus B(x,r'')\right)\cap \pi\right]\\
		\varphi_{k}(x),&x\in B(x,r'').
	\end{cases}
	\]
	It is easy check that $\tilde{\phi}_{k}$ is a Lipschitz map, by Kirszbraun's
	theorem, see for example \cite[2.10.43 Kirszbraun's theorem]{Federer:1969}, we can get a Lipschitz map 
	$\phi_{k}:\mathbb{R}^{n}\to\mathbb{R}^{n}$ such that 
	\[
	\phi_{k}\vert_{B( x,r'')}=\varphi_{k}\vert_{B(x,r'')}
	\]
	and
	\[
	\phi_{k}\vert_{B(x,r')^{c}}=\id_{B(x,r')^{c}},
	\]
	and 
	\[
	\phi_{k}\vert_{\left( B(x,r') \setminus B(x,r'')\right)\cap \pi}=
	\id_{\left( B(x,r') \setminus B(x,r'')\right)\cap \pi}.
	\]
	By the construction of $\phi_{k}$, we have that
	\[
	\phi_{k}(S_{k,r''})=\tilde{\phi}_{k}(S_{k,r''})\subset\partial D_{\pi,r''}.
	\]
	Since 
	\[
 \varphi_{x,\rho,r}(E_{k})\cap
	\left(B(x,r')\setminus B(x,r'')\right) \subset \left(B(x,r')\setminus
	B(x,r'')\right)\cap \pi
	\]
we have that 
	\[
	\phi_{k}\left( \varphi_{x,\rho,r}(E_{k})\cap
	B(x,r')\setminus B(x,r'') \right)\subset \left(B(x,r')\setminus
	B(x,r'')\right)\cap \pi.
	\]
	Thus
	\begin{align*}
	\phi_{k}(\varphi_{x,\rho,r}(E_{k}\cap
	B(x,r)))&=\phi_{k}(\varphi_{x,\rho,r}(E_{k})\cap B(x,r))\\
	&\subset\varphi_{x,\rho,r}
	(E_{k})\cap \left(B(x,r)\setminus B\left(x,r''\right)\right)\\
	&\subset\varphi_{x,\rho,r}\left( E_{k}\cap B(x,r)\setminus B(x,r'') \right).
\end{align*}
	Since $\H^{d}(E)<\infty$, we have that $\H^{d}(E\cap \partial B(x,r))=0$ for
	almost everywhere $r\in (0,r''(\varepsilon'',x))$, 
	if we take any $r\in (0,r''(\varepsilon'',x))$
	with $\H^{d}(E\cap \partial B(x,r))=0$ and $r<\delta$, then we have the following
	inequality:
	\begin{align*}
		\H^{d}\left(E\cap \overline{B(x,r)}\right)&=\H^{d}(E\cap B(x,r))\\
		&\leq\liminf_{k\to+\infty}\H^{d}(E_{k}\cap B(x,r))\\
		&\leq\liminf_{k\to+\infty}M\H^{d}\left(\phi_{k}\circ\varphi_{x,\rho,r}
		(E_{k}\cap B(x,r))\right)\\
		&\leq\liminf_{k\to+\infty}M\H^{d}\left(\varphi_{x,\rho,r}(E_{k}\cap
		B(x,r)\setminus B(x,r'''))\right)\\
		&\leq\liminf_{k\to+\infty}M\H^{d}\left(\varphi_{x,\rho,r}(E_{k}\cap
		\overline{B(x,r)}\setminus B(x,r'''))\right)\\
		&\leq\liminf_{k\to+\infty}M\left( 2+\frac{3h}{\rho}
		\right)^{d}\H^{d}\left(E_{k}\cap \overline{B(x,r)}\setminus
		B(x,r''')\right)\\
		&\leq M\left( 2+\frac{3h}{\rho}\right)^{d}\limsup_{k\to+\infty}
		\H^{d}\left(E_{k}\cap \overline{B(x,r)}\setminus B(x,r''')\right)\\
		&\leq M\left( 2+\frac{3h}{\rho}\right)^{d}\cdot M
		\H^{d}\left(E\cap \overline{B(x,r)}\setminus B(x,r''')\right)\\
		&\leq M^{2}\left( 2+\frac{3h}{\rho} \right)^{d}\varepsilon\H^{d}\left(E\cap
		\overline{B(x,r)}\right)\\
		&\leq M^{2}3^{d}\varepsilon\H^{d}\left(E\cap \overline{B(x,r)}\right).
	\end{align*}
	This is a contradiction since $M^{2}3^{d}\varepsilon<1$ and
	$\H^{d}\left(E\cap \overline{B(x,r)}\right)>0$.

	Since $F\in\mathfrak{F}$, by the definition, we have that 
	\[
	\mathbb{F}_{\kappa}(D_{\pi,r''})\leq \mathbb{F}_{\kappa}(S_{k,r''})+\varepsilon(x,r'')(r'')^{d}.
	\]

	Since $E$ is a $d$-rectifiable set and $\H^{d}(E)<+\infty$, the function
	$f:E\to G(n,d)$ defined by $f(x)=T_{x}E$ is $\H^{d}$-measurable. By Lusin's
	theorem, see for example~\cite[2.3.5. Lusin's theorem]{Federer:1969}, we can
	find a closed set $N\subset E$ with
	$\H^{d}(E\setminus N)<\varepsilon$ such that $f$ restricted to $N$
	is continuous.
	We put $E'''=(E''\cap N)$, then $E'''\subset E$ and
	\[
	\H^{d}(E\setminus E''')<\varepsilon,
	\]
	by Lemma 15.18 in~\cite{Mattila:1995}, we have that for $\H^{d}$-a.e. $x\in
	E'''$,
	\[
	T_{x}E'''=T_{x}N=T_{x}E.
	\]
	The map $\tilde{f}:E'''\to \mathbb{R}^{n}\times G(n,d)$ given by
	$\tilde{f}(x)=(x,T_{x}E)$ is continuous. Since $F$ is continuous, thus the function
	$F\circ\tilde{f}:E'''\to\mathbb{R}$ is continuous, for any $x\in E'''$, we
	can find $r(\varepsilon,x)>0$ such that 
	\[
	(1-\varepsilon)F(x,T_{x}E)\leq F(y,T_{y}E)\leq(1+\varepsilon)F(x,T_{x}E),
	\]
	for any $y\in E'''\cap B(x,r(\varepsilon,x))$. Thus, for all
	$0<r<r(\varepsilon,x)$, 
	\[
	(1-\varepsilon)\mathbb{F}_{\kappa}(T_{x}E\cap	B(x,r))\leq
	\mathbb{F}_{\kappa}(E'''\cap B(x,r))\leq(1+\varepsilon)\mathbb{F}_{\kappa}(T_{x}E\cap	B(x,r)).
	\]

	For any $x\in\mathbb{R}^{n}$, there exists $r'''(\varepsilon,x)>0$ such that
	$\varepsilon(x,r)<\varepsilon$ for all $0<r<r'''(\varepsilon,x)$.
	We put
	\[
	r(x)=\min(r(\varepsilon,x),r'(\varepsilon',x),r''(\varepsilon'',x),
	r'''(\varepsilon,x),\delta), \text{ for }x\in E''',
	\]
	then 
	\[
	\left\{B(x,r)\midd x\in E''', 0<r<r(x),
	\H^{d}(E\cap\partial B(x,r))=0 \right\}
	\]
	is a Vitali covering of $E'''$, so we can find a countable family of balls
	$(B_{i})_{i\in J}$ such that
	\[
	\H^{d}\left(E'''\setminus\bigcup_{i\in J}B_{i}\right)=0,
	\]
	and
	\[
	\sum_{i\in J}(r_{i})^{d}<\H^{d}(E''')+\varepsilon.
	\]

	We choose a finite set $I\subset J$ such that 
	\[
	\H^{d}\left(E'''\setminus\bigcup_{i\in I}B_{i}\right)<\varepsilon.
	\]

	We assume that $B_{i}=B(x_{i},r_{i})$. We put 
	\[
	\varphi=\prod_{i\in I}\varphi_{x_{i},\rho,r_{i}}.
	\]
	Since $\varphi\vert_{B_{i}}=\varphi_{x_{i},\rho,r_{i}}\vert_{B_{i}}$, we
	have that $\varphi\vert_{B(x_{i},r_{i}''')}=\id_{B(x_{i},r_{i}''')}$ and
	\[
	\varphi(E_{k})\cap B(x_{i},r_{i}'')\setminus
	B(x_{i},r_{i}''')\subset\varphi(E_{k}\cap B(x_{i},r_{i})\setminus
	B(x_{i},r_{i}'''))
	\]
	and
	\[
	\pi_{i}\cap B=\pi_{i}\cap\left( \left( B(x_{i},r_{i})\setminus
	B(x_{i},r_{i}'') \right)\cup\left( B(x_{i},r_{i}'')\setminus 
	B(x_{i},r_{i}''') \right)\cup\left( B(x_{i},r_{i}''') \right) \right),
	\]
	so we get that
	\[
	\mathbb{F}_{\kappa}(E''')=\sum_{i\in J}\mathbb{F}_{\kappa}(E'''\cap B_{i})\leq \sum_{i\in
	I}\mathbb{F}_{\kappa}(E'''\cap B_{i})+(\sup F)\varepsilon.
	\]
	For any $i\in I$,
	\begin{align*}
		\mathbb{F}_{\kappa}(E'''\cap B_{i})&\leq(1+\varepsilon)\mathbb{F}_{\kappa}(\pi_{i}\cap
		B_{i})\\
		&\leq(1+\varepsilon)\Big(\mathbb{F}_{\kappa}\left(\pi_{i}\cap B(x_{i},r_{i})
		\setminus B(x_{i},r_{i}'') \right)+\mathbb{F}_{\kappa}(\pi_{i}\cap
		B(x_{i},r_{i}''))\Big)\\
		&\leq(1+\varepsilon)\Big( 
		\mathbb{F}_{\kappa}(S_{k,r''})+ \varepsilon(x,r_{i}'')(r_{i}'')^{d}+(\sup
		F)(r_{i}^{d}-(r_{i}'')^{d})\Big)\\
		&\leq (1+\varepsilon) \mathbb{F}_{\kappa}(S_{k,r''})+2\varepsilon(r_{i}'')^{d}+
		2(\sup F)\left((r_{i}^{d}-(r_{i}'')^{d})\right)\\
		&\leq(1+\varepsilon)\mathbb{F}_{\kappa}(E_{k}\cap B(x_{i},r_{i}'''))\\
		&\quad+(1+\varepsilon)\mathbb{F}_{\kappa}(\varphi(E_{k}\cap B(x_{i},r_{i})\setminus
		B(x_{i},r_{i}''')))\\
		&\quad+\left(2\varepsilon +2(\sup F)\left(1-\left(
		1-\frac{2\rho}{3}\right)^{d}\right)\right)r_{i}^{d}\\
		&\leq(1+\varepsilon)\mathbb{F}_{\kappa}(E_{k}\cap B(x_{i},r_{i}'''))\\
		&\quad+2(\sup F)(\Lip
		\varphi)^{d}\H^{d}(E_{k}\cap B(x_{i},r_{i})\setminus B(x_{i},r_{i}'''))\\
		&\quad+\left(2\varepsilon +2(\sup F)\cdot
		\frac{\varepsilon}{2}\right)r_{i}^{d}.
	\end{align*}
	Hence
	\begin{align*}
		\mathbb{F}_{\kappa}(E''')&\leq\sum_{i\in I}\mathbb{F}_{\kappa}(E'''\cap B_{i})+(\sup
		F)\varepsilon\\
		&\leq(1+\varepsilon)\mathbb{F}_{\kappa}(E_{k})+\left(2\varepsilon 
		+(\sup F)\varepsilon\right)(\H^{d}(E''')+ \varepsilon)+(\sup
		F)\varepsilon\\
		&\quad+2(\sup F)(\Lip
		\varphi)^{d}\sum_{i\in I}\H^{d}\left(E_{k}\cap
		\overline{B(x_{i},r_{i})}\setminus B(x_{i},r_{i}''')\right),
	\end{align*}
	thus
	\begin{align*}
		\mathbb{F}_{\kappa}(E''')&\leq\liminf_{k\to+\infty}(1+\varepsilon)\mathbb{F}_{\kappa}(E_{k})
		+\left(2\varepsilon +(\sup
		F)\varepsilon\right)(\H^{d}(E''')+\varepsilon)+(\sup F)\varepsilon\\
		&\quad+2(\sup F)(\Lip \varphi)^{d}\liminf_{k\to+\infty}
		\sum_{i\in I}\H^{d}\left(E_{k}\cap \overline{B(x_{i},r_{i})}
		\setminus B(x_{i},r_{i}''')\right)\\
		&\leq\liminf_{k\to+\infty}(1+\varepsilon)\mathbb{F}_{\kappa}(E_{k})+
		(2\varepsilon+(\sup F)\varepsilon)(\H^{d}(E''')+\varepsilon)+(\sup
		F)\varepsilon\\
		&\quad+2(\sup F)(\Lip\varphi)^{d}M\sum_{i\in I}\H^{d}\left(E\cap
		\overline{B(x_{i},r_{i})}\setminus B(x_{i},r_{i}''')\right)\\
		&\leq\liminf_{k\to+\infty}(1+\varepsilon)\mathbb{F}_{\kappa}(E_{k})+
		(2\varepsilon+(\sup F)\varepsilon)(\H^{d}(E''')+\varepsilon)+(\sup
		F)\varepsilon\\
		&\quad+2(\sup F)(\Lip\varphi)^{d}M\varepsilon\H^{d}(E),
	\end{align*}
	and 
	\begin{align*}
		\mathbb{F}_{\kappa}(E)&=\mathbb{F}_{\kappa}(E''')+\mathbb{F}_{\kappa}(E\setminus E''')\\
		&\leq\mathbb{F}_{\kappa}(E''')+(\sup F)\H^{d}(E\setminus E''')\\
		&\leq(1+\varepsilon)\liminf_{k\to+\infty}\mathbb{F}_{\kappa}(E_{k})\\
		&\quad+\Big( \left( 2+\sup F+2\sup F (\Lip \varphi)^{d}M
		\right)(\H^{d}(E)+\varepsilon) +2\sup F\Big)\varepsilon.
	\end{align*}
	We can let $\varepsilon$ tend to 0, we get that 
	\[
	\mathbb{F}_{\kappa}(E)\leq\liminf_{k\to+\infty}\mathbb{F}_{\kappa}(E_{k}).
	\]
\end{proof}

The following proposition is taken from~\cite[3.1 Proposition]{Pauw:2007a}.
\begin{proposition}\label{prop:HLS}
	Let $B\subset\mathbb{R}^{n}$ be a compact subset. Suppose that for
	$j=1,2,\ldots$, $S_{j}\subset\mathbb{R}^{n}$ is a compact set with
	$B\subset S_{j}$, and that $S_{j}$ converge in Hausdorff distance to a
	compact set $S\subset\mathbb{R}^{n}$. Let $L\subset\check{H}_{k-1}(B;G)$ be 
	a subgroup such that $L\subset\ker\check{H}_{k-1}(i_{B,S_{j}})$.
	Then $L\subset \ker\check{H}_{k-1}(i_{B,S})$. 
\end{proposition}
The proof of the proposition is essentially the same as the proof of 
Proposition 3.1 in \cite{Pauw:2007a}, so we omit the proof.

\begin{theorem}
	Suppose that $0<d<n$ and that $F\in\mathfrak{F}$ is integrand.
	Then there is a positive constant $M>0$ such that for all open bounded domain
	$U\subset\mathbb{R}^n$, for all closed $d$-rectifiable set $E\subset U$ and 
	for all $\epsilon>0$, we can build a $n$-dimensional complex $\mathcal{S}$ 
	and a Lipschitz map $\phi\colon\mathbb{R}^n\rightarrow\mathbb{R}^n$ 
	satisfying the following properties:
	\begin{enumerate}[{\indent\rm (1)}]
		\item $\phi\vert_{\mathbb{R}^n\setminus U}=\id_{\mathbb{R}^n\setminus U}$ and $\Vert\phi-\id_{\mathbb{R}^n}\Vert_\infty\leq\epsilon$;
		\item $\mathcal{R}(\mathcal{S})\geq M$, where $\mathcal{R}(\mathcal{S})$ 
			is the shape control of $\mathcal{S}$, for the definition 
			see~\cite[p.8]{Feuvrier:2009}
			or~\cite[D\'efinition 1.2.27]{Feuvrier:2008}; 
		\item $\phi(E)$ is contained in the union of $d$-skeleton of $\mathcal{S}$,
			and $\left\vert \mathcal{S}\right\vert\subset U$;
		\item $\mathbb{F}_{\kappa}(\phi(E))\leq (1+\epsilon)\mathbb{F}_{\kappa}(E)$.
	\end{enumerate}
\end{theorem}
This is only a small improvement over Theorem 4.3.17
in~\cite{Feuvrier:2008} and Theorem 3 in~\cite{Feuvrier:2009}, but the
proof is almost same as that of V. Feuvrier in~\cite{Feuvrier:2008,Feuvrier:2009}.
What we only need to change is following: In the proof of V. Feuvrier, the 
multiplicity function $h$ is a continuous bounded function, thus for all 
$\epsilon'>0$ and all $x\in U$ there exists $r_{\max}'(x)>0$ such that 
\begin{equation}\label{multiplicity}
\forall y\in B(x,r_{\max}'(x)),\ (1-\epsilon')h(x)\leq h(y)\leq
(1+\epsilon')h(x).
\end{equation}
But in our paper, we use the integrand $F$ instead of $h$, and the inequality
(\ref{multiplicity}) will not be available, but this not too bad. We consider the
function $f:E\to G(n,d)$ defined by $f(x)=T_{x}E$. It is easy to
see that $f$ is measurable. By Lusin's theorem, see for 
example~\cite[2.3.5.]{Federer:1969}, we can write $E=E'\sqcup
E''$ such that $f\vert_{E'}$ is continuous and $\H^{d}(E'')\leq
\epsilon'\H^{d}(E)$. Then for all $x\in E'$, there exists $r_{\max}'(x)>0$ such
that for all $y\in E'\cap B(x,r_{\max}'(x))$, 
\[
(1-\epsilon')F(x,T_{x}E)\leq F(y,T_{y}E)\leq (1+\epsilon)F(x,T_{x}E).
\]
The rest of proof of our theorem will be the same as that 
in~\cite{Feuvrier:2008,Feuvrier:2009}.

Using this theorem, we can prove the following lemma.

\begin{lemma}\label{lemma1}
	Suppose that $0<d<n$ and that $U\subset\mathbb{R}^n$. Suppose that
	$F\in\mathfrak{F}$ is an integrand. Then there is a positive 
	constant $M'>0$ depending only on $d$ and $n$ such that
	for all relatively closed $d$-rectifiable set $E\subset U$,
	for all relatively compact subset $V\subset{U}$ and for all $\epsilon>0$,
	we can find a $n$-dimensional complex $\mathcal{S}$ and a subset 
	$E''\subset U$ satisfying the following properties:
	\begin{enumerate}[{\indent \rm (1)}]
		\item $E''$ is a $\diam(U)$-deformation of $E$ over $U$ and by putting
			$W=\mathring{\left\vert \mathcal{S}\right\vert}$ we have $V\subset
			W\subset\overline{W}\subset U$ and 
			there is a $d$-dimensional skeleton $\mathcal{S}'$ of $\mathcal{S}$ such
			that $E''\cap\overline{W}=\left\vert\mathcal{S}'\right\vert$;
		\item $\mathbb{F}_{\kappa}(E'')\leq(1+\epsilon)\mathbb{F}_{\kappa}(E)$;
		\item there are $d+1$ complexes $\mathcal{S}^0,\ldots,\mathcal{S}^d$ such
			that $\mathcal{S}^{\ell}$ is contained in the $\ell$-skeleton of 
			$\mathcal{S}$ and there is a decomposition
			\[
			E''\cap W=E^d\sqcup E^{d-1}\sqcup\ldots\sqcup E^0,
			\]
			where for each $0\leq \ell \leq d$,
			\[
			E^{\ell}\in {\bf QM}(W^{\ell},M',\diam(W^{\ell}),\H^{\ell}),
			\]
			where 
			\begin{align*}
				\begin{cases}
					W^{d}=W\\
					W^{\ell-1}=W^{\ell}\setminus E^{\ell}
				\end{cases}
				&&
				\begin{cases}
					E^{d}=\left\vert\mathcal{S}^{d}\right\vert\cap W^{d}\\
					E^{\ell}=\left\vert\mathcal{S}^{\ell}\right\vert\cap W^{\ell}.
				\end{cases}
			\end{align*}
	\end{enumerate}
\end{lemma}
The proof of this lemma is also the same as the proof of the Lemme 5.2.6
in~\cite{Feuvrier:2008} or the Lemma 9 in~\cite{Feuvrier:2009}. Therefore we
omit the proof.

We now turn to prove the main result of this paper.
\begin{proof}[Proof of the Theorem 1.2]
	We claim that we can find a ball $B(0,R)$ and a sequence of compact
	sets $(E_{k})_{k\geq 1}$ such that $B\subset B(0,R)$, $E_{k}$ spans $L$ in
	\v{C}ech homology, $E_{k}\subset B(0,R)$ and 
	\[
	\mathbb{F}_{\kappa}(E_{k}\setminus B)\to m(B,G,L,F,\kappa).
	\]

	We take any sequence of compact sets $(E_{k}')_{k\geq
	1}$ in $\mathcal{F}$ such that 
	\[
	\mathbb{F}_{\kappa}(E_{k}'\setminus B)\to m(B,G,L,F,\kappa).
	\]
	We take 
	\[
	U_{k}'=\{ x\in B(0, R_{k})\mid \dist(x,B)>2^{-k}\},
	\]
	where 
	\[
	R_{k}>\max\{ k, R_{k-1}+1, \dist(0, E_{k}')+\diam(E_{k}')+1\}.
	\]
	By lemma~\ref{lemma1}, we can find a Lipschitz map
	$\phi_{k}':\mathbb{R}^{n}\to\mathbb{R}^{n}$ and a complex $\mathcal{S}_{k}$
	such that 
	\[
	\phi_{k}'\vert_{U_{k}'}=\id_{U_{k}'},\ U_{k-1}'\subset\left\vert
	\mathcal{S}_{k}\right\vert\subset U_{k}',\ E_{k}'\subset\left\vert
	\mathcal{S}_{k}\right\vert,
	\]
	and
	\[
	\phi_{k}'(E_{k}')\cap W_{k}'=F_{k}\sqcup F_{k}',
	\]
	where $W_{k}'=\mathring{\left\vert \mathcal{S}_{k}\right\vert}$ and 
	\[
	F_{k}\in {\bf QM}(W_{k}',M,\diam(W_{k}'),\H^{d}),
	\]
	and $F_{k}'$ is contained in the union
	of $(d-1)$-dimensional skeleton of $\mathcal{S}_{k}$.

	We now prove that $(F_{k})_{k\geq 1}$ is bounded, i.e. we can find a large
	ball $B(0,r)$ such that $B\cup(\cup_{k} F_{k})\subset B(0,r)$. Suppose not,
	that is, suppose that for any large number $r>R_{1}$ there exist $k>4r$ such that
	$F_{k}\setminus B(0,2r)\neq\emptyset$. If $x\in F_{k}\setminus B(0,2r)$, 
	we take a cube $Q$ centered at $x$ with
	$\diam(Q)=r$, then by using Proposition 4.1 in~\cite{DS:2000}, we have 
	that 
	\[
	\H^{d}(F_{k}\cap Q)\geq C^{-1}\diam(Q)^{d},
	\]
	where $C$ only depend on
	$n$ and $M$. If we take $r$ large enough, for example
	\[
	r^{d}>\frac{2C}{\min\{ \inf F, \kappa \}}(m(B,G,L,F,\kappa)+1),
	\]
	and take $k$ large enough such that
	$\mathbb{F}_{\kappa}(E_{k}')<m(B,G,L,F,\kappa)+1$,
	then 
	\begin{align*}
	C^{-1}r^{d}&\leq \H^{d}(F_{k}\cap Q)\\
	&\leq\frac{1}{\min\{ \inf F, \kappa \}}\mathbb{F}_{\kappa}(\phi_{k}'(E_{k}'))\\
	&\leq \frac{(1+2^{-k})}{\min\{ \inf F, \kappa \}}\mathbb{F}_{\kappa}(E_{k}')\\
	&<\frac{2}{\min\{ \inf F, \kappa \}}(m(B,G,L,F,\kappa)+1),
\end{align*}
	this is a contradiction. Thus $\cup_{k} F_{k}$ is bounded. It is easy
	to see that $\cup_{k}(\phi_{k}'(E_{k}')\cap W_{k}'^{c})$ is bounded, so we 
	can assume that both $B\cup(\cup_{k} F_{k})$ and 
	$\cup_{k}(\phi_{k}'(E_{k}')\cap W_{k}'^{c})$ are contained in a large ball 
	$B(0,R)$. We take map $\rho:\mathbb{R}^{n}\to\mathbb{R}^{n}$ defined by
	\[
	\rho(x)=\begin{cases}
		x,&x\in B(0,R)\\
		\frac{R}{\left\vert x\right\vert}x,&x\in B(0,R)^{c},
	\end{cases}
	\]
	$\rho$ is 1-Lipschitz map.
	We put $E_{k}=\rho\circ\phi_{k}'(E_{k}')$, then $E_{k}\in\mathcal{F}$, and
	\[
	E_{k}=(\phi_{k}'(E_{k}')\cap W_{k}'^{c})\cup F_{k}\cup \rho(F_{k}').
	\]
	Since $\H^{d}(F_{k}')=0$, we have that 
	\[
	\mathbb{F}_{\kappa}(E_{k}\setminus
	B)=\mathbb{F}_{\kappa}(\phi_{k}'(E_{k}')\setminus B)\leq
	(1+2^{-k})\mathbb{F}_{\kappa}(E_{k}'\setminus B),
	\]
	therefore 
	\[
	\mathbb{F}_{\kappa}(E_{k}\setminus B)\to m(B,G,L,F,\kappa),
	\]
	and $(E_{k})_{k\geq 1}$ is a sequence which we desire.

	If $\mathbb{F}_{\kappa}(E_{k}\setminus B)=0$ for some $k\geq 1$, 
	then $m(B,G,L,F,\kappa)=0$ and $E_{k}$ is a minimizer, we have nothing
	to prove. We now suppose that for all $k\geq 1$,
	$0<\mathbb{F}_{\kappa}(E_{k}\setminus B)<+\infty$. Thus
	$0<\H^{d}(E_{k}\setminus B)<+\infty$.

	We put 
	\[
	U=B(0,R+1)\setminus B,\ V_{k}=\{x\in B(0,R+1-2^{k})\mid \dist(x,B)>2^{-k}\}.
	\]

	By lemma~\ref{lemma1}, we can find polyhedral complexes $\mathcal{S}_{k}$,
	Lipschitz maps
	$\phi_{k}:\mathbb{R}^{n}\to\mathbb{R}^{n}$ and a constant $M'=M'(n,d)$ 
	such that
	\begin{enumerate}[{\indent\rm (1)}]
		\item $V_{k}\subset\left\vert\mathcal{S}_{k}\right\vert\subset V_{k+1}$,
			$\phi_{k}\vert_{V_{k+1}^{c}}=\id_{V_{k+1}^{c}}$,
			and there exists a $d$-dimensional skeleton $\mathcal{S}_{k}'$ of
			$\mathcal{S}_{k}$ such that
			$E_{k}''\cap W_{k}=\left\vert\mathcal{S}_{k}' \right\vert$, where
			$E_{k}''=\phi_{k}(E_{k})$ and
			$W_{k}=\left\vert\mathcal{S}_{k}\right\vert$;
		\item $\mathbb{F}_{\kappa}(E_{k}''\setminus B)\leq(1+2^{-k})\mathbb{F}_{\kappa}(E_{k}\setminus B)$;
		\item there exist complexes
			$\mathcal{S}_{k}^{0},\ldots,\mathcal{S}_{k}^{d}$ such that
			$\mathcal{S}_{k}^{\ell}$ is contained in the $\ell$-skeleton of
			$\mathcal{S}_{k}$ and there is a disjoint decomposition 
			\[
			E_{k}''\cap \mathring{W}_{k}=E_{k}^{d}\sqcup
			E_{k}^{d-1}\sqcup\cdots\sqcup E_{k}^{0},
			\]
			where for each $0\leq \ell \leq d$,
			\[
			E_{k}^{\ell}\in {\bf QM}(W_{k}^{\ell},M',\diam(W_{k}^{\ell}),\H^{\ell}),
			\]
			where 
			\begin{align*}
				\begin{cases}
					W_{k}^{d}=\mathring{W}_{k}\\
					W_{k}^{\ell-1}=W_{k}^{\ell}\setminus E_{k}^{\ell}
				\end{cases}
				&&
				\begin{cases}
					E^{d}=\left\vert\mathcal{S}_{k}^{d}\right\vert\cap W_{k}^{d}\\
					E^{\ell}=\left\vert\mathcal{S}_{k}^{\ell}\right\vert\cap W_{k}^{\ell},
				\end{cases}
			\end{align*}
			and $\mathring{W}_{k}$ is the interior of $W_{k}$.
	\end{enumerate}

	We note that for each $k$, $E_{k}''$ and $W_{k}$ are two compact subsets of
	$\mathbb{R}^{n}$, thus $E_{k}\cap W_{k}$ is a compact subset of
	$\mathbb{R}^{n}$. We may suppose that $E_{k}''\cap W_{k}\to E'$ in Hausdorff
	distance, passing to a subsequence if necessary. We put $E=E'\cup B$. We
	will show that $E$ is a minimizer.

	First of all, we show that $E$ spans $B$ in \v{C}ech homology, so
	$E\in\mathcal{F}$. Since $\phi_{k}$ is Lipschitz map and
	$\phi_{k}\vert_{V_{k+1}^{c}}=\id_{V_{k+1}^{c}}$, in particular,
	$\phi_{k}\vert_{B}=\id_{B}$, thus $E_{k}''=\phi_{k}(E_{k})$ spans $L$ 
	in \v{C}ech homology. Since $V_{k}\subset W_{k}\subset V_{k+1}$, we have
	that 
	\[
	B\subset \phi_{k}(E_{k})\setminus W_{k}\subset B(2^{-k}),
	\]
	where we denote by $B(\epsilon)$ denote the $\epsilon$-neighborhood of $B$.
	Thus $E_{k}''\setminus W_{k}\to B$ in Hausdorff distance, so 
	\[
	E_{k}''=(E_{k}''\cap W_{k})\cup (E_{k}''\setminus W_{k})\to E'\cup B=E.
	\]
	By proposition~\ref{prop:HLS}, we have that $E$ spans $L$ in \v{C}ech
	homology.

	Next, we will show that $\mathbb{F}_{\kappa}^{d}(E\setminus
	B)=m(B,G,L,F,\kappa)$.

	Passing to a subsequence if necessary, we may assume that 
	\[
	E_{k}^{d}\to E^{d}\text{ in } U, \text{ for } 0\leq \ell \leq d,
	\]
	For any $0\leq\ell\leq d$, we put
	\[
	U^{\ell}=U\setminus\bigcup_{\ell<\ell'\leq d}E^{\ell'},
	\]
	we assume that $E_{k}^{\ell}\to E^{\ell}$ in $U$. Then 
	\[
	E\setminus B=\bigcup_{0\leq\ell\leq d} E^{\ell}.
	\]

	Since
	\[
	E_{k}^{d}\in {\bf QM}(W_{k}^{d},M',\diam(W_{k}^{d}),\H^{d}),
	\]
	we can apply the Theorem~\ref{thm:lsc}, and get that 
	\[
	\mathbb{F}_{\kappa}(E^{d}\cap W_{k}^{d})\leq\liminf_{m\to\infty}\mathbb{F}_{\kappa}(E_{m}^{d}\cap
	W_{k}^{d})\leq \liminf_{m\to\infty}\mathbb{F}_{\kappa}(E_{m}^{d}).
	\]
	Since $V_{k}\subset W_{k}\subset V_{k+1}$ and $W_{k}^{d}=\mathring{W}_{k}$,
	we have 
	\[
	\bigcup_{k} W_{k}^{d}=\bigcup_{k} V_{k}=U,
	\]
	thus 
	\[
	\mathbb{F}_{\kappa}(E^{d})\leq\liminf_{m\to\infty}\mathbb{F}_{\kappa}(E_{m}^{d}).
	\]

	For any $0\leq\ell\leq d$, for any $\varepsilon>0$, we put
	$U_{\varepsilon}^{d}=B(0,R+1-\varepsilon)\cap U^{d}$ and 
	\[
	U_{\varepsilon}^{\ell}=\left\{	x\in B(0,R+1-\varepsilon)\midd \dist
	\left(x,\bigcup_{\ell<\ell'\leq d}E^{\ell'}\right)>\varepsilon\right\}.
	\]
	Then $U_{\varepsilon_{1}}^{\ell}\subset U_{\varepsilon_{2}}^{\ell}$ for any
	$0<\varepsilon_{2}<\varepsilon_{1}$, and 
	\[
	\bigcup_{\varepsilon>0} U_{\varepsilon}^{\ell}=U^{\ell}.
	\]
	Since $E_{k}^{\ell}\to E^{\ell}$ in $U$, we have that $E_{k}^{\ell}\cap
	U_{\varepsilon}\to E^{\ell}\cap U_{\varepsilon}$ in $U_{\varepsilon}$.
	We will show that for any $\varepsilon>0$, there exists $k_{\varepsilon}$
	such that for $k\geq k_{\varepsilon}$,
	\[
	E_{k}^{\ell}\cap U_{\varepsilon}^{\ell}\in {\bf QM}(U_{\varepsilon}^{\ell},
	M',\diam(U_{\varepsilon}^{\ell}),\H^{\ell}).
	\]

	Indeed, for any	$\varepsilon>0$, we can find $k_{\varepsilon}$ such that 
	$U_{\varepsilon}^{\ell}\subset W_{k}^{\ell}$. We prove this fact by
	induction on $\ell$.

	First, we take a positive integer $k_{\varepsilon}$ such that $2^{-k_{\varepsilon}}<\varepsilon$, 
	then $U_{\varepsilon}^{d}\subset W_{k}^{d}$ for any $k\geq k_{\varepsilon}$.

	Next, we suppose that there is an integer $k_{\varepsilon}$ such that
	$U_{\varepsilon}^{\ell}\subset W_{k}^{\ell}$ for $k\geq k_{\varepsilon}$.
	Since $E_{k}^{\ell}\to E^{\ell}$ in $U^{\ell}$ and
	\[
	W_{k}^{\ell-1}=W_{k}^{\ell}\setminus E_{k}^{\ell},\ 
	U_{\varepsilon}^{\ell}=\left\{ x\in U_{\varepsilon}^{\ell}\midd \dist\left(x,
	E^{\ell}\right)>\varepsilon \right\},
	\]
	we can find $k_{\varepsilon}'$ such that $U_{\varepsilon}^{\ell-1}\subset
	W_{k}^{\ell-1}$ for $k\geq k_{\varepsilon}'$.

	Since $U_{\varepsilon}^{\ell}\subset W_{k}^{\ell}$ and 
	\[
	E_{k}^{\ell}\in{\bf QM}(W_{k}^{\ell},M',\diam(W_{k}^{\ell}),\H^{\ell}),
	\]
	we get that 
	\[
	E_{k}^{\ell}\cap U_{\varepsilon}^{\ell}\in{\bf
	QM}(U_{\varepsilon}^{\ell},M',\diam(U_{\varepsilon}^{\ell}),\H^{\ell}).
	\]

	For any $\delta>0$, we put $\Omega_{\delta}=\{ x\in U_{\varepsilon}\mid
	\dist(x,U_{\varepsilon}^{c})\geq 10\delta \}$. $E_{k}^{\ell}\cap
	\Omega_{\delta}$ is a compact set, and $\{ B(x,\delta)\mid x\in
	E_{k}^{\ell}\cap\Omega_{\delta} \}$ is an open covering of
	$E_{k}^{\ell}\cap\Omega_{\delta}$, we can find a finitely many balls $\{
	B(x_{i},\delta)\}_{i\in I}$ which is a covering of $E_{k}^{\ell}\cap
	\Omega_{\delta}$, by the 5-covering lemma, see for example the Theorem 2.1 in 
	\cite{Mattila:1995}, we can find a subset $J\subset I$ such
	that $B(x_{j_{1}},\delta)\cap B(x_{j_{2}},\delta)=\emptyset$ for
	$j_{1},j_{2}\in J$ with $j_{1}\neq j_{2}$, and 
	\[
	\bigcup_{i\in I} B(x_{i},\delta)\subset \bigcup_{j\in J} B(x_{j},5\delta).
	\]
	Since $B(x_{j_{1}},\delta)\cap B(x_{j_{2}},\delta)=\emptyset$ for
	$j_{i},j_{2}\in J$, we have that
	\[
	\mathcal{L}^{n}(U_{\varepsilon})\geq \sum_{j\in
	J}\mathcal{L}^{n}(B(x_{j},\delta)),
	\]
	thus 
	\[
	\#J\leq\frac{\mathcal{L}^{n}(U_{\varepsilon})}{\omega_{n}\delta^{n}}.
	\]
	By the Proposition 4.1 in~\cite{DS:2000}, we have that 
	\[
	C^{-1}(5\delta)^{\ell}\leq\H^{\ell}(E_{k}^{\ell}\cap B(x_{j},5\delta))\leq
	C(5\delta)^{\ell},
	\]
	so
	\[
	\H^{\ell}(E_{k}^{\ell}\cap \Omega_{\delta})\leq\sum_{j\in
	J}\H^{\ell}(E_{k}^{\ell}\cap B(x_{j},5\delta))\leq\sum_{j\in
	J}C(5\delta)^{\ell}\leq\omega_{n}^{-1}
	\mathcal{L}^{n}(U_{\varepsilon})5^{d}\delta^{\ell-n}C.
	\]

	Applying the theorem 3.4 in~\cite{David:2003}, we get that
	\[
	\H^{\ell}(E^{\ell}\cap
	\Omega_{\delta})\leq\liminf_{k\to\infty}\H^{\ell}(E_{k}^{\ell}\cap
	\Omega_{\delta})\leq\omega_{n}^{-1}
	\mathcal{L}^{n}(U_{\varepsilon})5^{d}\delta^{\ell-n}C,
	\]
	and $\dim_{\H}E^{\ell}\cap \Omega_{\delta}\leq\ell$, hence $\dim_{\H}E^{\ell}\leq \ell$, thus $\H^{d}(E^{\ell})=0$.

	we get that
	\begin{align*}
		\mathbb{F}_{\kappa}(E\setminus B)&=\mathbb{F}_{\kappa}(E^{d})\\
		&\leq\liminf_{k\to\infty}\mathbb{F}_{\kappa}(E_{k}^{d})\\
		&\leq\liminf_{k\to\infty}\mathbb{F}_{\kappa}(E_{k}''\setminus B)\\
		&\leq\liminf_{k\to\infty}(1+2^{-k})\mathbb{F}_{\kappa}(E_{k}\setminus B)\\
		&=\liminf_{k\to\infty}\mathbb{F}_{\kappa}(E_{k}\setminus B)\\
		&=m(B,G,L,F,\kappa).
	\end{align*}

	Since $E\in\mathcal{F}$, we have that 
	\[
	\mathbb{F}_{\kappa}(E\setminus B)\geq m(B,G,L,F,\kappa),
	\]
	therefore 
	\[
	\mathbb{F}_{\kappa}(E\setminus B)=m(B,G,L,F,\kappa).
	\]

\end{proof}

\bibliography{\jobname}

\begin{bibdiv}
\begin{biblist}

\bib{Almgren:1968}{article}{
      author={{Almgren, Jr.}, F.~J.},
       title={Existence and regularity almost everywhere of solutions to
  elliptic variational problems among surfaces of varying topological type and
  singularity structure},
        date={1968},
     journal={Ann. of Math.},
      volume={87},
      number={2},
       pages={321\ndash 391},
}

\bib{Almgren:1974}{article}{
      author={{Almgren, Jr.}, F.~J.},
       title={The structure of limit varifolds associated with minimizing
  sequences of mappings},
        date={1974},
     journal={Sympos. Math.},
      volume={14},
       pages={413\ndash 428},
}

\bib{Almgren:1976}{book}{
      author={{Almgren, Jr.}, F.~J.},
       title={Existence and regularity almost everywhere of solutions to
  elliptic variational problems with constraints},
      series={Memoirs of the American Mathematical Society},
        date={1976},
      volume={165},
}

\bib{David:2003}{article}{
      author={David, G.},
       title={Limits of almgren quasiminimal sets},
        date={2003},
     journal={Proceedings of the conference on Harmonic Analysis, Mount
  Holyoke, A.M.S. Contemporary Mathematics series},
      volume={320},
       pages={119\ndash 145},
}

\bib{DS:2000}{book}{
      author={David, G.},
      author={Semmes, S.},
       title={Uniform rectifiability and quasiminimizing sets of arbitrary
  codimension},
      series={Memoirs of the A.M.S. 687},
        date={2000},
      volume={144},
}

\bib{Douglas:1931}{article}{
      author={Douglas, Jesse},
       title={Solutions of the problem of {Plateau}},
        date={1931},
     journal={Transactions of the American Mathematical Society},
      volume={33},
       pages={263\ndash 321},
}

\bib{Dowker:1952}{article}{
      author={Dowker, C.~H.},
       title={Homology groups of relations},
        date={1952},
     journal={Ann. of Math.},
      volume={56},
      number={1},
       pages={84\ndash 95},
}

\bib{ES:1952}{book}{
      author={Eilenberg, S.},
      author={Steenrod, N.},
       title={Foundations of algebraic topology},
   publisher={Princeton},
        date={1952},
}

\bib{Federer:1969}{book}{
      author={Federer, H.},
       title={Geometric measure theory},
   publisher={Springer-Verlag, New York},
        date={1969},
}

\bib{FF:1960}{article}{
      author={Federer, H.},
      author={Fleming, W.},
       title={Normal and integral currents},
        date={1960},
     journal={Ann. of Math.},
      volume={72},
      number={3},
       pages={458\ndash 520},
}

\bib{Feuvrier:2008}{thesis}{
      author={Feuvrier, V.},
       title={Un r\'esultat d'existence pour les ensembles minimaux par
  optimisation sur des grilles poly\'edrales},
        type={Ph.D. Thesis},
        date={2008},
}

\bib{Feuvrier:2009}{article}{
      author={Feuvrier, V.},
       title={Condensation of polyhedric structures onto soap films},
        date={2009},
        note={Preprint, \url{http://arxiv.org/abs/0906.3505}},
}

\bib{DMS}{article}{
      author={G.~Dal~Maso, J. M.~Morel},
      author={Solimini., S.},
       title={A variational method in image segmentation: Existence and
  approximation results},
        date={1992},
     journal={Acta. Math.},
      volume={1},
      number={168},
       pages={89\ndash 151},
}

\bib{Mattila:1995}{book}{
      author={Mattila, P.},
       title={Geometry of sets and measures in euclidean spaces: Fractals and
  rectifiability},
      series={Cambridge Studies in Advanced Mathematics},
   publisher={Cambridge U. Press},
        date={1995},
      volume={44},
}

\bib{Morgan:1989}{article}{
      author={Morgan, F.},
       title={Size-minimizing rectifiable currents},
        date={1989},
     journal={Invent. Math.},
      volume={96},
      number={2},
       pages={333\ndash 348},
}

\bib{Pauw:2007a}{article}{
      author={Pauw, T.~De},
       title={Size minimizing surfaces with integral coefficients},
        date={2007},
     journal={Ann. Sci. Ecole Norm. Sup.},
      volume={42},
}

\bib{PH:2003}{article}{
      author={Pauw, T.~De},
      author={Hardt, R.},
       title={Size minimization and approximation problems},
        date={2003},
     journal={Calc. Var. Partial Differential Equations},
      volume={17},
       pages={405\ndash 442},
}

\bib{Reifenberg:1960}{article}{
      author={Reifenberg, E.~R.},
       title={Solution of the plateau problem for $m$-dimensional surfaces of
  varying topological type},
        date={1960},
     journal={Bull. Amer. Math. Soc.},
      volume={66},
}

\end{biblist}
\end{bibdiv}
\end{document}